\documentclass[12pt]{article}
\usepackage{latexsym,amsfonts,amssymb,mathrsfs,float}
\usepackage{dsfont}
\usepackage{amsthm}
\usepackage[numbers]{natbib}
\usepackage{amsmath}
\usepackage{indentfirst}

\usepackage[dvips]{color}
\usepackage{colordvi,multicol}
\allowdisplaybreaks

\newtheorem{thm}{Theorem}[section]

\newtheorem{lem}[thm]{Lemma}
\newtheorem{pro}[thm]{Proposition}

\newtheorem{rem}[thm]{Remark}

\numberwithin{equation}{section}

\begin{document}

\title{\bf Favorite Downcrossing Sites of One-Dimensional Simple Random Walk}
 \author{ Chen-Xu Hao$^1$, Ze-Chun Hu$^1$, Ting Ma$^1$ and Renming Song$^2$\\ \\
 {\small $^1$College of Mathematics, Sichuan  University,
 Chengdu 610065, China}\\
 {\small 476924193@qq.com; zchu@scu.edu.cn; matingting2008@scu.edu.cn }\\ \\
  {\small $^2$Department of Mathematics,
University of Illinois, Urbana, IL 61801, USA}\\
 {\small rsong@illinois.edu}}

\maketitle
\date{}
\noindent{\bf Abstract:}\ Random walk is a very important Markov process and has important applications in many fields.  For a one-dimensional simple symmetric random walk $(S_n)$, a site $x$ is called a favorite downcrossing site at time $n$ if its downcrossing local time at time $n$ achieves the maximum among all sites. In this paper, we study the cardinality of the favorite downcrossing site set, and will show that with probability 1 there are only finitely many times at which there are at least four favorite downcrossing sites and three favorite downcrossing sites occurs infinitely often. Some related open questions will be introduced.

\smallskip

\noindent {\bf Keywords: }\  Random walk, favorite site, favorite edge, favorite downcrossing site, local time.

\section{Introduction and main result}

Let $(S_n)_{n\in\mathbb{N}}$ be a one-dimensional simple symmetric random walk with $S_0=0$. Following T\'{o}th and Werner \cite{TW97}, we define the number of upcrossings and downcrossings of the site $x$ by time $n$ to be
\begin{eqnarray*}
&&\xi_U(x,n):=\#\{0<k\leq n:S_k=x,S_{k-1}=x-1\},\\
&&\xi_D(x,n):=\#\{0<k\leq n:S_k=x,S_{k-1}=x+1\},
\end{eqnarray*}
respectively. Hereafter, $\# D$ denotes the cardinality of the set $D$.  Define the  local times of site $x$ and edge $x$ by
\begin{eqnarray*}
&&\xi(x,n):=\xi_U(x,n)+\xi_D(x,n),\\
&&L(x,n):=\xi_U(x,n)+\xi_D(x-1,n),
\end{eqnarray*}
respectively, where edge $x$ is between points $x-1$ and $x$.

\begin{itemize}
\item A site $x$ is called a favorite (most visited) site of the random walk at time $n$ if
$$
\xi(x,n)=\max_{y\in{\mathbb{Z}}}\xi(y,n).
$$
The set of favorite sites of the random walk at time $n$ is denoted by $\mathcal{K}(n)$. $(\mathcal{K}(n))_{n\ge 1}$ is called the favorite site process of the one-dimensional simple symmetric random walk. We say that {\it three favorite sites} occurs at time $n$ if $\#\mathcal{K}(n)=3$.

\item An edge $x$ is called a favorite  edge of the random walk at time $n$ if
$$
L(x,n)=\sup_{y\in{\mathbb{Z}}}L(y,n).
$$
The set of favorite edges of the random walk at time $n$ is denoted by $\mathcal{E}(n)$. $(\mathcal{E}(n))_{n\ge 1}$ is called the favorite edge process of the one-dimensional simple symmetric random walk.
We say that {\it three favorite edges} occurs at time $n$ if $\#\mathcal{E}(n)=3$.

\item  A site $x$ is called a favorite downcrossing site of the random walk at time $n$ if
$$
\xi_D(x,n)=\max_{y\in{\mathbb{Z}}}\xi_D(y,n).
$$
The set of favorite downcrossing sites of the random walk at time $n$ is denoted by $\mathcal{K}_D(n)$. $(\mathcal{K}_D(n))_{n\ge 1}$ is called the favorite downcrossing site process of the one-dimensional simple symmetric random walk. We say that {\it three favorite downcorssing sites} occurs at time $n$ if $\#\mathcal{K}_D(n)=3$.

\end{itemize}

In \cite[Theorem 1.1]{HHMS21}, we showed that with probability 1 three favorite edges occurs infinitely often, which complements the result in \cite{TW97} (in which the authors showed that eventually there are no more than three favoriate edges), and disproves a conjecture mentioned in
\cite[Remark 1, p. 368]{TW97}. In the proof of \cite[Theorem 1.1]{HHMS21},
we used the transience of the favorite downcrossing site process to show the transience of the favorite edge process. In fact, from \cite{HHMS21},
we can see that there is a close relation between favorite edges and favorite downcrossing sites, which can be regarded as one motivation of this paper.

In this paper, we study the cardinality of the favorite downcrossing site set, and
the main result of this paper is the following

\begin{thm}\label{mainthm-0}
For one-dimensional simple symmetric random walk, with probability 1 there are only finitely many times at which there are at least four favorite downcrossing sites and three favorite downcrossing sites occurs infinitely often.
\end{thm}

For the related problem of the number of favorite sites of one-dimensional simple symmetric random walks, there are many more references (see Shi and T\'{o}th \cite{ST00} for an overview).
This problem was posed by Erd\H{o}s and R\'{e}v\'{e}sz \cite{ER84, ER87, ER91, Re90}. T\'{o}th \cite{To01} proved that there are no more than three favorite sites eventually. Ding and Shen \cite{DS18} proved that with probability 1 three favorite sites occurs infinitely often.

In addition to the study of the number of favorite sites,
a series of papers focus on the asymptotic behavior of favorite sites, see \cite{Ba22, BG85, CRS00, CS98,  LS04}. For papers on favorite sites of simple random walks in  higher dimensions, we refer to \cite{Ab15, De05, DPRZ01, Ok16}.
In addition, there are a number of papers
on favorite sites for other processes
including Brownian motion, symmetric stable process, L\'{e}vy processes, random walks in random environments  and so on, see \cite{BES00, CDH18, Ei89, Ei97, EK02,  HS00, HS15, KL95, Le97, LXY19, Ma01}.

In \cite[Proposition 2.4]{HHMS21}, we proved the following result.

\begin{pro}\label{pro-1.2}
If $x\in \mathcal{E}(n)$, then $x-1\in \mathcal{K}_D(n)$.
\end{pro}

By Proposition \ref{pro-1.2}  and \cite[Theorem 1.1]{HHMS21},
 it is obvious that in order to prove Theorem \ref{mainthm-0}, we need only to prove that with probability 1 there are only finitely many times at which there are at least four favorite downcrossing sites.

For $r\geq 1$, we define
\begin{eqnarray*}
f(r):=\#\left\{n\geq 1: S_n\in \mathcal{K}_D(n),S_n=S_{n-1}-1,  \#\mathcal{K}_D(n)=r\right\},
\end{eqnarray*}
the number of times at which a new favorite downcrossing site appears, tied with  $r-1$ other favorite downcrossing sites. By definition $f(r)\ge f(r+1)$. Then in order to complete the proof of Theorem \ref{mainthm-0}, it is enough to show

\begin{thm}\label{mainthm}
$\mathbb{E}(f(4))<+\infty$.
\end{thm}

 Our proof is inspired by \cite{To01} and \cite{TW97}.  The rest of the paper is organized as follows. In Section 2, we present some preliminaries. In Section 3, we give the proof of Theorem \ref{mainthm}.  In the final section,
 we give some remarks and introduce several open questions.

Throughout the paper, we denote by $C$ some positive constant whose value may vary even within one proof.

\section{Preliminaries}

We define the inverse local times by
\begin{eqnarray*}
T_U(x,k):=\min\{n\geq 1:\xi_U(x,n)=k\}\quad \mbox{and}\quad T_D(x,k):=\min\{n\geq 1:\xi_D(x,n)=k\}.
\end{eqnarray*}
We rewrite $f(4)$ as
$$
f(4)=\sum_{x\in\mathbb{Z}}d(x),
$$
where
\begin{eqnarray*}
d(x)&:=&\sum_{n=1}^{\infty}\textbf{1}_{\{S_{n-1}=x+1,S_n=x,x\in\mathcal{K}_D(n),\#\mathcal{K}_D(n)=4\}}\nonumber\\
&=&\sum_{n=1}^{\infty}\sum_{k=1}^{\infty}\textbf{1}_{\{T_D(x,k)=n,x\in\mathcal{K}_D(n),\#\mathcal{K}_D(n)=4\}}\nonumber\\
&=&\sum_{k=1}^{\infty}\textbf{1}_{\{x\in\mathcal{K}_D(T_D(x,k)),
\#\mathcal{K}_D(T_D(x,k))=4\}}.
\end{eqnarray*}
It follows that
\begin{eqnarray}\label{Ef4}
\mathbb{E}(f(4))=\sum_{x=1}^{+\infty}\mathbb{E}(d(-x))+\sum_{x=0}^{+\infty}\mathbb{E}(d(x)),
\end{eqnarray}
where
\begin{eqnarray*}
\mathbb{E}(d(x))=\sum_{k=1}^{\infty}\mathbb{P}\left(x\in\mathcal{K}_D(T_D(x,k)),
\#\mathcal{K}_D(T_D(x,k))=4\right).
\end{eqnarray*}

\subsection{Branching processes and the Ray-Knight representation}

In the remainder of this paper, we denote by $Y_n$ a critical Galton-Watson branching process with geometric offspring distribution, and by $Z_n,R_n$ two related critical Galton-Watson branching processes with immigration in different ways.
More precisely, the definitions of these processes are as follows: Let $(X_{n, i})_{n, i}$
be i.i.d. geometric variables with mean 1, that is, for all $k\geq 0,P(X_{n,i}=k)=\frac{1}{2^{k+1}}$. We recursively define
\begin{eqnarray}\label{GWbp}
Y_{n+1}=\sum_{i=1}^{Y_n}X_{n,i},\ Z_{n+1}=\sum_{i=1}^{Z_n+1}X_{n,i},\ R_{n+1}=1+\sum_{i=1}^{R_n}X_{n,i}.
\end{eqnarray}
Then $Y_n,Z_n$ and $R_n$ are Markov chains with state space
$\mathbb{N}$ and transition probabilities as follows:
\begin{equation}\label{TP}
\mathbb{P}(Y_{n+1}=j|Y_n=i)=\pi(i,j):=\left\{
\begin{aligned}
&\delta_0(j), &\text{if~$i=0$},\\
&2^{-i-j}\frac{(i+j-1)!}{(i-1)!j!}, &\text{if~$i>0$},
\end{aligned}\right.
\end{equation}
\begin{eqnarray}
\mathbb{P}(Z_{n+1}=j|Z_n=i)=\rho(i,j):=\pi(i+1,j),\\
\mathbb{P}(R_{n+1}=j|R_n=i)=\rho^{*}(i,j):=\pi(i,j-1).\label{2.5}
\end{eqnarray}

These three processes are very useful tools for describing
the local time process of $(S_n)_{n\ge 0}$ in the space variable, taken at certain stopping times. Let $x$ be a fixed integer. We have the following two cases:

{\bf Case 1.} $x\leq -1$. We define the following three processes:

\begin{itemize}
\item $(R_n^{(h)})_{0\leq n\leq -x-1}$ is a Markov chain with transition probabilities $\rho^*(i,j)$ and initial state $R^{(h)}_0=h$;

\item $(Y_n^{(h-1)})_{0\leq n< \infty}$ is a Markov chain with transition probabilities $\pi(i,j)$ and initial state $Y^{(h-1)}_{0}=h-1$;

\item $(Y_n^{'(h)})_{0\leq n< \infty}$ is a Markov chain with transition probabilities $\pi(i,j)$ and initial state $Y_0^{'(h)}=R_{-x-1}^{(h)}$.
\end{itemize}

\noindent We assume that the three processes are independent, except that $(Y_n^{'(h)})$ starts from
$R_{-x-1}^{(h)}$.
We patch the above three processes together to a single process as follows:
\begin{equation}\label{RK1}
\Delta_x^{(h)}(y):=\left\{
\begin{array}{ll}
Y_{y+1}^{'(h)},&\text{if~$-1\le y< +\infty$},\\
R_{y-x}^{(h)}, &\text{if~$x\leq y\le -1$},\\
Y_{x-y}^{(h-1)}, &\text{if~$-\infty<y\le x-1$}.
\end{array}\right.
\end{equation}
By the Ray-Knight theorem (\cite[Theorem~1.1]{Kn63}), we know that
\begin{eqnarray}\label{RK}
(\xi_D(y,T_D(x,h)),y\in\mathbb{Z})\overset{{\rm law}}{=}(\Delta_{x}^{(h)}(y),y\in\mathbb{Z}).
\end{eqnarray}

{\bf Case 2.} $x\ge 0$.  Define the following three processes:

\begin{itemize}

\item $(Z_n^{(h-1)})_{0
\leq n\leq x}$ is a Markov chain with transition probabilities $\rho(i,j)$ and initial state $Z_0^{(h-1)}=h-1$;

\item $(Y_n^{(h)})_{0\leq n<\infty}$ is a Markov chain with transition probabilities $\pi(i,j)$ and initial state $Y_{0}^{(h)}=h$;

\item $(Y_n^{'(h-1)})_{0 \leq n<\infty }$ is a Markov chain with transition probabilities $\pi(i,j)$ and initial state $Y_{0}^{'(h-1)}=Z_{x}^{(h-1)}$.
\end{itemize}

In this case, we patch the three processes together by
\begin{equation}\label{RK3}
\Gamma_x^{(h)}(y):=\left\{
\begin{array}{ll}
Y_{-y}^{'(h-1)}, &\text{if~$y\le 0$},\\
Z_{x-y}^{(h-1)}, &\text{if~$0\le y\le x-1$},\\
Y_{y-x}^{(h)},&\text{if~$y\ge x$}.
\end{array}\right.
\end{equation}
By \cite[Theorem~1.1]{Kn63}, we know that
\begin{eqnarray}\label{RK4}
(\xi_D(y,T_D(x,h)),y\in\mathbb{Z})\overset{{\rm law}}{=}(\Gamma_{x}^{(h)}(y),y\in\mathbb{Z}).
\end{eqnarray}

\begin{rem}
(i) In case 1, if $x=-1$, (\ref{RK1})should be modified as follows:
\begin{eqnarray*}
\Delta_x^{(h)}(y):=\left\{
\begin{array}{ll}
Y_{y+1}^{'(h)},&\text{if~$-1\le y< +\infty$},\\
Y_{x-y}^{(h-1)}, &\text{if~$-\infty<y\le -2$},
\end{array}\right.
\end{eqnarray*}
where $Y_{0}^{'(h)}=h$.

(ii) In case 2, if $x=0$, (\ref{RK3}) should  modified as follows:
\begin{eqnarray*}
\Gamma_x^{(h)}(y):=\left\{
\begin{array}{ll}
Y_{-y}^{'(h-1)}, &\text{if~$y\le -1$},\\
Y_{y-x}^{(h)},&\text{if~$y\ge 0$},
\end{array}\right.
\end{eqnarray*}
where $Y_{0}^{'(h-1)}=h-1$.

(iii) In order to save space, we will not discuss the above two special cases separately in the corresponding proof below.
\end{rem}

Note that the initial state of $(Y_n)$ (resp. $(R_n)$) is $h-1$ (resp. $h$) in the case 1, but the initial state of $(Y_n)$ (resp. $(Z_n)$) is $h$ (resp. $h-1$) in the case 2. For all the notations above, the initial state of other processes $(R_n), (Z_n), (Y_n ') $ other than $(Y_n) $ will be omitted below. We will also use conditional probability $P(\cdot|Y_0=h)$ to indicate the initial state.

\subsection{$\mathbb{E}(d(x))$ under the Ray-Knight representation}

Let $Y_n, Y'_n, Z_n$ and $R_n$ be the branching processes
defined in the above subsection.
For $h\in\mathbb{N}$, define the following stopping times:
\begin{eqnarray}\label{tingshi}
&&\sigma_h:=\min\{n\geq 1: Y_{n}\geq h\},\ \sigma_h':=\min\{n\geq 1: Y_{n}'\geq h\};\nonumber\\
&&\tau_h:=\min\{n\geq 1:Z_n\geq h\},\ \tau_h':=\min\{n\geq 1:R_n\geq h\}.
\end{eqnarray}

{\bf Case 1:}
Let $x\le -1,h\ge 1, p\geq 0$ be integers.  We define
\begin{eqnarray}
&&A_{h,p}:=\left\{\max_{1\le n<\infty}Y_n^{(h-1)}\le h,\#\{1\le n<\infty:Y_n^{(h-1)}=h\}=p\right\},\label{Ray-Knight representation-a}\\
&&A_{h,p}':=\left\{\max_{1\le n<\infty}Y_n'\le h,\#\{1\le n<\infty:Y_n'=h\}=p\right\},\label{Ray-Knight representation-b}\\
&&B_{x,h,p}:=\left\{\max_{1\le n\le -x-1}R_n\le h,\#\{1\le n\le -x-1:R_n=h\}=p\right\}.\label{Ray-Knight representation-c}
\end{eqnarray}
Note that the event $A_{h,p}'$ depends on $x$ through the initial state of $(Y_n')$. In the following, we will use conditional probabilities to indicate the initial state.

By (\ref{RK}), we know  that
\begin{eqnarray}\label{Ray-Knight representation-u(x)}
\mathbb{E}(d(x))&=&\sum_{p+q+r=3}\sum_{h=1}^{+\infty}\sum_{l=0}^{h}\mathbb{P}(A_{h,p}|Y_0=h-1)\cdot \mathbb{P}(B_{x,h,q}\cap\{R_{-x-1}=l\}|R_0=h)\nonumber\\
&&\quad\quad\quad\quad\quad\quad\quad \cdot \mathbb{P}(A_{h,r}'|Y_0'=l).
\end{eqnarray}
For convenience, we substitute ``$-x$" for ``$x$". Then we have
\begin{eqnarray}\label{Ray-Knight representation-u(x)01}
\sum_{x=1}^{\infty}\mathbb{E}\big(d(-x)\big)&\le&\sum_{p+q+r=3}
\sum_{h=1}^{+\infty}\mathbb{P}(A_{h,p}|Y_0=h-1)
\cdot \left(\sum_{x=1}^{\infty}\mathbb{P}(B_{-x,h,q}|R_0=h)\right)\nonumber\\
&&\quad\quad\quad\quad\quad\cdot \left(\sup_{l\in[0,h]}\mathbb{P}(A_{h,r}'|Y_{0}'=l)\right).
\end{eqnarray}

{\bf Case 2:}
Let $x\ge 0, h\ge 1, p\geq 0$ be integers.
We define
\begin{eqnarray}
C_{h,p}&:=&\left\{\max_{1\le n<\infty}Y_n^{(h)}\le h,\#\{1\le n<\infty:Y_n^{(h)}=h\}=p\right\},\label{Ray-Knight representation03-a}\\
C_{h,p}'&:=&\left\{\max_{1\le n<\infty}Y_n'\le h,\#\{1\le n<\infty:Y_n'=h\}=p\right\},\label{Ray-Knight representation03-b}\\
D_{x,h,p}&:=&\left\{\max_{1\le n\le x}Z_n\le h,\#\{1\le n\le x:Z_n=h\}=p\right\}.\label{Ray-Knight representation03-c}
\end{eqnarray}
Note that the event $C_{h,p}'$ depends on $x$ through the initial state of $(Y_n')$.
In the following, we will use conditional probabilities to indicate the initial state.

By (\ref{RK4}), we have that
\begin{eqnarray}\label{Ray-Knight representation-d(x)}
\mathbb{E}(d(x))&=&\sum_{p+q+r=3}\sum_{h=1}^{+\infty}\sum_{l=0}^{h}\mathbb{P}(C_{h,p}|Y_0=h)\cdot \mathbb{P}(D_{x,h,q}\cap\{Z_{x}=l\}|Z_0=h-1)\nonumber\\
&&\quad\quad\quad\quad\quad\quad \cdot \mathbb{P}(C_{h,r}'|Y_0'=l),
\end{eqnarray}
which implies that
\begin{eqnarray}\label{Ray-Knight representation-d(x)02}
\sum_{x=0}^{\infty}\mathbb{E}\big(d(x)\big)&\le&\sum_{p+q+r=3}
\sum_{h=1}^{+\infty}\mathbb{P}(C_{h,p}|Y_0=h)
\cdot \left(\sum_{x=0}^{\infty}\mathbb{P}(D_{x,h,q}|Z_0=h-1)\right)\nonumber\\
&&\quad\quad\quad\quad\quad  \cdot\left(\sup_{l\in[0,h]}\mathbb{P}(C_{h,r}'|Y_{0}'=l)\right).
\end{eqnarray}

\subsection{Some lemmas}

In order to give the proof of Theorem \ref{mainthm} in next section, we need several lemmas.

\begin{lem}(\cite[Side-lemma 3]{To01})\label{Side-lemma 3}
There exists a constant $C<\infty$ such that for any $0\le k\le h$, it holds that
\begin{eqnarray*}
\mathbb{P}(\sigma_h=\infty|Y_0=k)\le \frac{h-k}{h}+Ch^{-\frac12}.
\end{eqnarray*}
\end{lem}

By the above lemma and the definitions of the two processes $(Y_n)_{n\geq 0}$ and $(Y_n')_{n\geq 0}$,  we have

\begin{lem}\label{Side-lemma 3'}
There exists a constant $C<\infty$ such that for any $0\le k\le h$, it holds that
\begin{eqnarray*}
\mathbb{P}(\sigma_h'=\infty|Y_0'=k)\le \frac{h-k}{h}+Ch^{-\frac12}.
\end{eqnarray*}
\end{lem}

\begin{lem}(\cite[Side-lemma 4]{To01})\label{Side-lemma 4}
There exists a constant $C<\infty$ such that for any $0\le k\le h$, it holds that
\begin{eqnarray*}
\mathbb{E}(\tau_h|Z_0=k)\le (h-k)+Ch^{\frac12}.
\end{eqnarray*}
\end{lem}

By following \cite[Corollary,  p.  500]{To01},
we can get a similar result on $\tau_h'$ as follows.

\begin{lem}\label{E tau'}
There exists a constant $C<\infty$ such that for any
integer $h\ge 1$,
the following upper bound holds:
\begin{eqnarray}\label{E tau'-1}
\mathbb{E}(\tau_h'|R_0=h)\leq Ch^{\frac12}.
\end{eqnarray}
\end{lem}
\begin{proof} One can check that  $(R_t-t)_{t\geq 0}$ is a martingale. Then by the optional stopping theorem,  we have
\begin{eqnarray}\label{tau'}
\mathbb{E}(\tau_h'|R_0=h)=\mathbb{E}(R_{\tau_h'}|R_0=h)-h
=-h+\sum_{u=h}^{\infty}\mathbb{P}(R_{\tau_h'}\ge u|R_0=h).
\end{eqnarray}

We claim that
\begin{eqnarray}\label{tau'1}
\mathbb{P}(R_{\tau_h'}\ge u|R_0=h)\le\dfrac{\sum_{w=u}^{\infty}\rho^*(h,w)}{\sum_{v=h}^{\infty}\rho^*(h,v)}.
\end{eqnarray}
Note that, for $1\le h\le u$,
\begin{eqnarray}\label{tau'1-a}
\mathbb{P}(R_{\tau_h'}=u|R_0=h)&=&
\sum_{l=1}^{h-1}\mathbb{P}(R_{\tau_h'}=u,R_{\tau_h'-1}=l|R_0=h)\nonumber\\
&=&\sum_{l=1}^{h-1}\mathbb{P}(R_{\tau_h'-1}=l|R_{0}=h)\mathbb{P}(R_{\tau_h'}
=u|R_{\tau_h'-1}=l,R_{0}=h)\nonumber\\
&=&\sum_{l=0}^{h-1}\mathbb{P}(R_{\tau_h'-1}=l|R_{0}=h)\cdot
\dfrac{\rho^*(l,u)}{\sum_{v=h}^{\infty}\rho^*(l,v)}.
\end{eqnarray}
By (\ref{TP}) and (\ref{2.5}), we know that for any $0\le l<h\le v$, it holds that
\begin{eqnarray*}
\dfrac{\rho^*(l+1,v)}{\rho^*(l,v)}=\dfrac{\pi(l+1,v-1)}{\pi(l,v-1)}=\frac{l+v-1}{2l}<\frac{l+v}{2l}
=\dfrac{\pi(l+1,v)}{\pi(l,v)}=\dfrac{\rho^*(l+1,v+1)}{\rho^*(l,v+1)}.
\end{eqnarray*}
It follows that for any $0\le l<h\le v<w$,
\begin{eqnarray*}
&&\rho^*(l+1,v)\cdot\rho^*(l,w)<\rho^*(l,v)\rho^*(l+1,w),\\
&&\rho^*(l,v)=\frac{2l}{l+v-1}\rho^*(l+1,v)\leq \rho^*(l+1,v),
\end{eqnarray*}
which imply  that for any $0\le l<h\le u$,
\begin{eqnarray*}
\sum_{v=h}^{\infty}\rho^*(l+1,v)\cdot\sum_{w=u}^{\infty}\rho^*(l,w)<\sum_{v=h}^{\infty}\rho^*(l,v)\cdot\sum_{w=u}^{\infty}\rho^*(l+1,w).
\end{eqnarray*}
Then, by (\ref{tau'1-a}), we get
\begin{eqnarray*}
\mathbb{P}(R_{\tau_h'}\ge u|R_{0}=h)\le\max_{l'\in[0,h]}\dfrac{\sum_{w=u}^{\infty}\rho^*(l',w)}{\sum_{v=h}^{\infty}\rho^*(l',v)}
\le\dfrac{\sum_{w=u}^{\infty}\rho^*(h,w)}{\sum_{v=h}^{\infty}\rho^*(h,v)},
\end{eqnarray*}
i.e. (\ref{tau'1}) holds.

By (\ref{tau'1}) and some explicit computations,
we get that
\begin{eqnarray*}\label{lem-2.4-2.20}
\sum_{u=h}^{\infty}\mathbb{P}(R_{\tau_h'}\ge u|R_{0}=h)\le \dfrac{\sum_{u=h}^{\infty}\sum_{w=u}^{\infty}\rho^*(h,w)}
{\sum_{v=h}^{\infty}\rho^*(h,v)}
=\frac{\sum_{v=h}^{\infty}\rho^*(h,v)v}{\sum_{w=h}^{\infty}\rho^*(h,w)}
\leq h+Ch^{1/2},
\end{eqnarray*}
which together with (\ref{tau'}) implies (\ref{E tau'-1}).
\end{proof}

For $(Z_n)_{n\ge 0}$, we have the following result.

\begin{lem} (\cite[Lemma 3.1(ii)]{TW97})\label{Branching process}
For any $\varepsilon>0$, there exists a constant $C(\varepsilon)<\infty$ such that for
any integer $h>0$,
\begin{eqnarray*}
\mathbb{P}(Z_{\tau_h}=h|Z_0=h)\le C(\varepsilon) h^{-\frac12+\varepsilon}.
\end{eqnarray*}
\end{lem}

For $(R_n)_{n\ge 0}$, by following \cite[Lemma 3.1(ii)]{TW97}, we can obtain the following result.

\begin{lem}\label{Branching process1}
For any $\varepsilon>0$, there exists a constant $C(\varepsilon)<\infty$ such that for
any integer $h>0$,
\begin{eqnarray}\label{lem-2.7-a}
P(R_{\tau_h'}=h|R_0=h)\le C(\varepsilon)h^{-\frac12+\varepsilon}.
\end{eqnarray}
\end{lem}

\begin{proof}
We divide the proof into three steps:

{\bf Step 1.} Let $\frac12\le\alpha\le 1$, $h-h^{\alpha}\le z\le h$, and $k\le\sqrt{h}$. We have
\begin{eqnarray*}
&&\dfrac{\rho^*(z,h+k)}{\rho^*(z,h)}=\dfrac{\pi(z,h+k-1)}{\pi(z,h-1)}\\
&&=2^{-k}
\dfrac{(h+z-1)(h+z+1)\cdots(h+z+k-2)}{h(h+1)\cdots(h+k-1)}\\
&&\ge 2^{-k}\cdot\dfrac{(h+z-1)^k}{(h+k-1)^k}\\
&&\ge2^{-k}\cdot\dfrac{(2h-h^{\alpha}-1)^k}{(h+\sqrt{h}-1)^k}\\
&&\ge(1-2h^{\alpha-1})^k,
\end{eqnarray*}
where in the last inequality we used
\begin{eqnarray*}
\frac12\cdot\dfrac{2h-h^{\alpha}-1}{h+\sqrt{h}-1}\ge 1-2h^{\alpha-1}\Leftrightarrow \left(3h^{\alpha}-2\sqrt{h}\right)+4\left(h^{\alpha-\frac12}-h^{\alpha-1}\right)+1\ge 0.
\end{eqnarray*}
Then when $h$ is large enough such that $1-2h^{\alpha-1}\in (0,1)$,  we have
\begin{eqnarray}\label{lem-2.7-b}
&&\dfrac{P(R_0\ge h|R_{-1}=z)}{P(R_0=h|R_{-1}=z)}=\sum_{k=0}^{\infty}\dfrac{\rho^*(z,h+k)}{\rho^*(z,h)}
\geq \sum_{k=0}^{[\sqrt{h}]}\dfrac{\rho^*(z,h+k)}{\rho^*(z,h)}\nonumber\\
&&\ge\sum_{k=0}^{[\sqrt{h}]}(1-2h^{\alpha-1})^k=
\dfrac{1-(1-2h^{\alpha-1})^{\sqrt{h}}}{2h^{\alpha-1}}
\ge\dfrac{1-(1-2h^{\alpha-1})^{\sqrt{h}-1}}{2h^{\alpha-1}}.
\end{eqnarray}
Notice that
\begin{eqnarray}\label{lem-2.7-c}
(1-2h^{\alpha-1})^{\sqrt{h}-1}\le (1-2h^{-\frac12})^{\sqrt{h}-1}\overset{h\to\infty}{\longrightarrow}e^{-2}<\frac12.
\end{eqnarray}
By (\ref{lem-2.7-b}) and (\ref{lem-2.7-c}), there exists $h_0>0$ such that for any $h>h_0$,
\begin{eqnarray}\label{lem-2.7-d}
\mathbb{P}\big(R_1=h|\{R_0=z\}\cap\{R_1\ge h\}\big)\le 4h^{\alpha-1},
\end{eqnarray}
where we used the following inequality
$$
\mathbb{P}\big(R_1=h|\{R_0=z\}\cap\{R_1\ge h\}\big)\leq
\dfrac{P(R_1=h|R_0=z)}{P(R_1\ge h|R_0=z)},
$$
which can be checked by the definition of conditional probability.

{\bf Step 2.} For every $u\ge 0$, define
$$\theta_u':=\inf\{n\ge 0:R_n\le u\}.$$
Recall $(R_t-t)_{t\geq 0}$ is a martingale. For $0\le u\le h$, by the optional stopping theorem, we have
\begin{eqnarray}\label{lem-2.7-e}
\mathbb{E}(R_{\tau_h'}|R_0=h)-\mathbb{E}(\tau_h'|R_0=h)=
\mathbb{E}(R_{\tau_h'\land\theta_u'}|R_0=h)-\mathbb{E}(\tau_h'\land\theta_u'|R_0=h).
\end{eqnarray}
By (\ref{lem-2.7-e}) and Lemma \ref{E tau'}, we have
\begin{eqnarray}\label{lem-2.7-f}
\mathbb{E}\big((R_{\tau_h'}-R_{\theta_u'})\textbf{1}_{\{\theta_u'\leq \tau_h'\}}|R_0=h\big)\le \mathbb{E}(\tau_h'|R_0=h)\le Ch^{\frac12}.
\end{eqnarray}
By $R_{\tau_h'}\ge h, R_{\theta_u'}\le u$ and (\ref{lem-2.7-f}),  we get that
\begin{eqnarray}\label{lem-2.7-g}
\mathbb{P}(\theta_u'\leq \tau_h'|R_0=h)\le\dfrac{Ch^{\frac12}}{h-u}.
\end{eqnarray}
In particular, for $\beta\in[\frac12,1]$,
\begin{eqnarray}\label{lem-2.7-h}
\mathbb{P}(\theta_{h-h^{\beta}}'\leq \tau_h'|R_0=h)\le Ch^{\frac12-\beta}.
\end{eqnarray}

{\bf Step 3.} Fix $\varepsilon>0$ and choose integer $N$ such that $N>\dfrac{1}{2\varepsilon}$. For $i=\{0,1,2,\ldots,N\}$, define
\begin{eqnarray*}
u_i:=h-h^{\frac{N+i}{2N}},
\end{eqnarray*}
and
\begin{eqnarray*}
\Delta_0:=[u_0,h],\ \Delta_i:=[u_i,u_{i-1}),i=1,2,\ldots,N.
\end{eqnarray*}
Then
\begin{eqnarray}\label{Branching01}
\mathbb{P}(R_{\tau_h'}=h|R_0=h)=\sum_{i=0}^{N}
\mathbb{P}\left(\{R_{\tau_h'}=h\}\cap\{R_{\tau_h'-1}\in\Delta_i\}|R_0=h\right).
\end{eqnarray}

$(1)$ For $i\in\{1,\ldots,N\}$, by the strong Markov property, (\ref{lem-2.7-d}) and (\ref{lem-2.7-h}), we get
\begin{eqnarray}\label{Branching02}
&&\mathbb{P}\big(\{R_{\tau_h'}=h\}\cap\{R_{\tau_h'-1}\in\Delta_i\}|R_0=h\big)\nonumber\\
&=&\mathbb{P}\big(\{R_{\tau_h'}=h\}\cap\{R_{\tau_h'-1}\in\Delta_i\}
\cap\{\theta_{u_{i-1}}'\le\tau_h'\}|R_0=h\big)\nonumber\\
&\le&\mathbb{P}(\theta_{u_{i-1}}'\le\tau_h'|R_0=h)\cdot
\sup_{z\in\Delta_i}\mathbb{P}\left(R_1=h|\{R_0=z\}\cap\{R_1\ge h\}\right)\nonumber\\
&\le&ch^{-\frac{N+i-1}{2N}+\frac12}\cdot h^{-\frac{N-i}{2N}}=Ch^{-\frac{1}{2}+\frac{1}{2N}}\nonumber\\
&\le&Ch^{-\frac12+\varepsilon}.
\end{eqnarray}

$(2)$ For $i=0$, $\mathbb{P}\big(\{R_{\tau_h'}=h\}\cap\{R_{\tau_h'-1}\ge u_0\}|R_0=h\big)$ can be bounded directly. Take $z\in[h-\sqrt{h},h]$, then for all $k\in[0,\sqrt{h}]$,
\begin{eqnarray*}
\dfrac{\rho^*(z,h+k)}{\rho^*(z,h)}\ge (1-2h^{-\frac12})^k\ge (1-2h^{-\frac12})^{\sqrt{h}}\ge C.
\end{eqnarray*}
Then summing over $k$,
\begin{eqnarray*}
\dfrac{\mathbb{P}(R_1\ge h|R_0=z)}{\mathbb{P}(R_1=h|R_0=z)}\ge C\sqrt{h},
\end{eqnarray*}
which implies that
\begin{eqnarray*}
\dfrac{\mathbb{P}(R_1=h|R_0=z)}{\mathbb{P}(R_1\ge h|R_0=z)}\le \dfrac{C}{\sqrt{h}}.
\end{eqnarray*}
Again, by the strong Markov property, we have
\begin{eqnarray}\label{Branching03}
&&\mathbb{P}\Big(\{R_{\tau_h'}=h\}\cap\{R_{\tau_{h-1}'}\in[h-h^{\frac12},h)\}|R_0=h\Big)\nonumber\\
&&\le\mathbb{P}\big(R_{\tau_h'-1}\in[h-h^{\frac12},h)|R_0=h\big)\cdot
\sup_{z\in[h-h^{\frac12},h)} \mathbb{P}(R_1=h|R_0=z,R_1\ge h)\nonumber\\
&&\le\sup_{z\in[h-h^{\frac12},h)} \mathbb{P}(R_1=h|R_0=z,R_1\ge h)\nonumber\\
&&\le Ch^{-\frac12}.
\end{eqnarray}

Finally, combining (\ref{Branching01}), (\ref{Branching02}) and (\ref{Branching03}), we get this lemma.
\end{proof}

For $(Z_n)_{n\ge 0}$ with initial state $Z_0=h-1$, we have the following result.

\begin{lem}\label{Branching process h-1}
For any $\varepsilon>0$, there exists a constant $C(\varepsilon)<\infty$ such that for any integer $h>0$,
\begin{eqnarray*}
\mathbb{P}(Z_{\tau_h}=h|Z_0=h-1)\le C(\varepsilon) h^{-\frac12+\varepsilon}.
\end{eqnarray*}
\end{lem}

\begin{lem}(\cite[Lemma 3.3(ii)]{TW97})\label{Lemma 3.3(2)}
For any $\varepsilon>0$, there exists a constant $C(\varepsilon)<\infty$ such that for
any integer $h>0$,
\begin{eqnarray*}
\mathbb{P}\left([\sigma_h<\infty]\cap[Y_{\sigma_h}=h]|Y_0=h\right)\le C(\varepsilon)h^{-\frac12+\varepsilon}.
\end{eqnarray*}
\end{lem}

By the above lemma and the definitions of the two processes $(Y_n)_{n\geq 0}$ and $(Y_n')_{n\geq 0}$,  we have

\begin{lem}\label{Lemma 3.3(2)'}
For any $\varepsilon>0$, there exists a constant $C(\varepsilon)<\infty$ such that for
any integer $h>0$,
\begin{eqnarray*}
\mathbb{P}\left([\sigma_h'<\infty]\cap[Y_{\sigma_h'}=h]|Y_0'=h\right)\le C(\varepsilon)h^{-\frac12+\varepsilon}.
\end{eqnarray*}
\end{lem}

Consider the process $(Y_n)_{n\geq 0}$ with initial state $h-1$, similarly as Lemma \ref{Lemma 3.3(2)}, we have

\begin{lem}\label{Lemma 3.3(2)h-1}
For any $\varepsilon>0$, there exists a constant $C(\varepsilon)<\infty$ such that for
any integer $h>0$,
\begin{eqnarray*}
\mathbb{P}\left([\sigma_h<\infty]\cap[Y_{\sigma_h}=h]|Y_0=h-1\right)\le C(\varepsilon)h^{-\frac12+\varepsilon}.
\end{eqnarray*}
\end{lem}

\section{Proof of Theorem \ref{mainthm}}

In this section, we give the proof of Theorem \ref{mainthm}. To this end, we need the following two lemmas, whose proofs are put at the end of this section.

\begin{lem}\label{proof heart1}
For any $\varepsilon>0$, there exists a finite constant $C<\infty$ such that:

(i) for any integer $p\ge 0$,
\begin{eqnarray*}
\mathbb{P}(A_{h,p}|Y_0=h-1)\le (Ch^{-\frac{1}{2}+\varepsilon})^{p}\cdot h^{-\frac12};
\end{eqnarray*}

(ii) for any integer $p\ge 0$,
\begin{eqnarray*}
\sum_{x=1}^{\infty}\mathbb{P}(B_{-x,h,p}|R_0=h)\le (Ch^{-\frac{1}{2}+\varepsilon})^{p}\cdot h^{\frac12};
\end{eqnarray*}

(iii)
for any integers $p\ge 1,0\le l\le h$ and $x\le -1$,
\begin{eqnarray*}
\mathbb{P}(A_{h,p}'|Y_0'=l)\le (Ch^{-\frac{1}{2}+\varepsilon})^p.
\end{eqnarray*}
\end{lem}

\begin{lem}\label{proof heart2}
For any $\varepsilon>0$, there exists a finite constant $C<\infty$ such that:

(i)
for any integer $p\geq 0$,
\begin{eqnarray*}
\sum_{x=1}^{\infty}\mathbb{P}(C_{x,p}|Y_0=h)\le (Ch^{-\frac12+\varepsilon})^{p}\cdot h^{-\frac12};
\end{eqnarray*}

(ii)
for any integer $p\geq 0$,
\begin{eqnarray*}
\sum_{x=1}^{\infty}\mathbb{P}(D_{x,h,q}|Z_0=h-1)\le (Ch^{-\frac12+\varepsilon})^{p}\cdot h^{\frac12};
\end{eqnarray*}

(iii)
for any integer $p\geq 1,0\leq l\leq h$ and $x\ge 0$,
\begin{eqnarray*}
\mathbb{P}(C_{h,p}'|Y_0'=l)\le C(h^{-\frac12+\varepsilon})^{p}.
\end{eqnarray*}
\end{lem}

\bigskip

\noindent {\bf Proof of Theorem \ref{mainthm}.}  By (\ref{Ef4}), it is enough to show that \begin{eqnarray}\label{proof-the-main-thm-a}
\sum_{x=1}^{\infty}\mathbb{E}(d(-x))+\sum_{x=0}^{\infty}\mathbb{E}(d(x))<\infty.
\end{eqnarray}

By Lemma \ref{proof heart1}, we can bound the term
\begin{eqnarray}\label{proof-the-main-thm-b}
\sum_{p+q+r=3}&&\mathbb{P}(A_{h,p}|Y_0=h-1)\cdot
\left(\sum_{x=1}^{\infty}\mathbb{P}(B_{-x,h,q}|R_0=h)\right)\cdot
\left(\sup_{l\in[0,h]}\mathbb{P}(A_{h,r}'|Y_0'=l)\right).
\end{eqnarray}
If $r=0$ and $p+q=3$, (\ref{proof-the-main-thm-b}) is controlled by
\begin{eqnarray*}
C(h^{-\frac12+\varepsilon})^{p}\cdot h^{-\frac12}\cdot(h^{-\frac12+\varepsilon})^{q}\cdot h^{\frac12}=Ch^{-\frac32+3\varepsilon}.
\end{eqnarray*}
If $r>0$ and $p+q+r=3$, (\ref{proof-the-main-thm-b}) is controlled by
\begin{eqnarray*}
C(h^{-\frac12+\varepsilon})^{p}\cdot h^{-\frac12}\cdot(h^{-\frac12+\varepsilon})^{q}\cdot h^{\frac12}\cdot (h^{-\frac12+\varepsilon})^{r}=Ch^{-\frac32+3\varepsilon}.
\end{eqnarray*}
Then by (\ref{Ray-Knight representation-u(x)01}), we get that
\begin{eqnarray}\label{proof-the-main-thm-c}
\sum_{x=1}^{\infty}\mathbb{E}(d(-x))<\infty.
\end{eqnarray}

By Lemma \ref{proof heart2}, we can bound the term
\begin{eqnarray}\label{proof-the-main-thm-d}
\sum_{p+q+r=3}\mathbb{P}(C_{h,p}|Y_0=h)
\cdot \left(\sum_{x=0}^{\infty}\mathbb{P}(D_{x,h,q}|Z_0=h-1)\right)\cdot\left(\sup_{l\in[0,h]}\mathbb{P}(C_{h,r}'|Y_{0}'=l)\right).
\end{eqnarray}
If $r=0$ and $p+q=3$, (\ref{proof-the-main-thm-d}) is controlled by
\begin{eqnarray*}
C(h^{-\frac12+\varepsilon})^{p+q}=Ch^{-\frac32+3\varepsilon}.
\end{eqnarray*}
If $r>0$ and $p+q+r=3$, (\ref{proof-the-main-thm-d}) is controlled by
\begin{eqnarray*}
C(h^{-\frac12+\varepsilon})^{p+q}\cdot (h^{-\frac12+\varepsilon})^{r}=Ch^{-\frac32+3\varepsilon}.
\end{eqnarray*}
Then by (\ref{Ray-Knight representation-d(x)02}), we get that
\begin{eqnarray}\label{proof-the-main-thm-e}
\sum_{x=0}^{\infty}\mathbb{E}(d(x))<\infty.
\end{eqnarray}
(\ref{proof-the-main-thm-c}) and (\ref{proof-the-main-thm-e}) imply (\ref{proof-the-main-thm-a}).\hfill\fbox

\bigskip

\noindent {\bf Proof of Lemma \ref{proof heart1}.} (i) For $p=0$,  by (\ref{Ray-Knight representation-a}) and  Lemma \ref{Side-lemma 3}, we have
\begin{eqnarray*}
\mathbb{P}(A_{h,0}|Y_0=h-1)=\mathbb{P}(\sigma_h=\infty|Y_0=h-1)\leq Ch^{-\frac{1}{2}}.
\end{eqnarray*}
In this case (i) holds. By (\ref{Ray-Knight representation-a}) and  Lemma \ref{Side-lemma 3}, we also have
\begin{eqnarray}\label{proof-lem-3.1-a}
\mathbb{P}(A_{h,0}|Y_0=h)=\mathbb{P}(\sigma_h=\infty|Y_0=h)\leq Ch^{-\frac{1}{2}}.
\end{eqnarray}
For $p\ge 1$, by the strong Markov property of $(Y_n)$, Lemma \ref{Lemma 3.3(2)}, Lemma \ref{Lemma 3.3(2)h-1} and (\ref{proof-lem-3.1-a}) we have
\begin{eqnarray*}
&&\mathbb{P}(A_{h,p}|Y_0=h-1)\\
&&=\mathbb{P}
\left(\{\sigma_h<\infty\}\cap\{Y_{\sigma_h}=h\}|Y_0=h-1\right)\cdot \mathbb{P}(A_{h,p-1}|Y_0=h)\\
&&=\cdots\\
&&=\mathbb{P}
\left(\{\sigma_h<\infty\}\cap\{Y_{\sigma_h}=h\}|Y_0=h-1\right)\cdot\left(\mathbb{P}
\left(\{\sigma_h<\infty\}\cap\{Y_{\sigma_h}=h\}|Y_0=h\right)\right)^{p-1}\\
 &&\quad\cdot \mathbb{P}(A_{h,0}|Y_0=h)\\
&&\leq (Ch^{-\frac{1}{2}+\varepsilon})^{p}\cdot h^{-\frac12},
\end{eqnarray*}
which implies (i).

(ii) For $p=0$, by (\ref{Ray-Knight representation-c}) and  Lemma \ref{E tau'}, we have
\begin{eqnarray}\label{proof-lem-3.1-b}
\sum_{x=1}^{\infty}\mathbb{P}(B_{-x,h,0}|R_0=h)=\sum_{x=1}^{\infty}\mathbb{P}(\tau_h'\ge x|R_0=h)=\mathbb{E}(\tau_h'|R_0=h)\leq Ch^{\frac{1}{2}}.
\end{eqnarray}
For $p\ge 1$, by the strong Markov property of $(R_n)$, we have
\begin{eqnarray*}
\sum_{x=1}^{\infty}\mathbb{P}(B_{-x,h,p}|R_0=h)&=&\mathbb{P}(R_{\tau_h'}=h|R_0=h)
\cdot\left(\sum_{x=1}^{\infty}\mathbb{P}(B_{-x,h,p-1}|R_0=h)\right)\\
&=&\cdots\\
&=&\left(\mathbb{P}(R_{\tau_h'}=h|R_0=h)\right)^p
\cdot\left(\sum_{x=1}^{\infty}\mathbb{P}(B_{-x,h,0}|R_0=h)\right).
\end{eqnarray*}
Then by (\ref{proof-lem-3.1-b}) and  Lemma \ref{Branching process1}, we get
\begin{eqnarray*}
\sum_{x=1}^{\infty}\mathbb{P}(B_{-x,h,p}|R_0=h)\le (Ch^{-\frac{1}{2}+\varepsilon})^p\cdot h^{\frac12}.
\end{eqnarray*}

(iii) By the definition of $A_{h,0}^{'}$ (i.e. (\ref{Ray-Knight representation-b})), and Lemma \ref{Side-lemma 3'}, we have
\begin{eqnarray}\label{proof-lemma-3.1-(iii)}
\mathbb{P}(A_{h,0}'|Y_0'=h)=\mathbb{P}(\sigma_h'=\infty|Y_0'=h)\leq Ch^{-\frac{1}{2}}.
\end{eqnarray}
For $p\geq 1$, by the strong Markov property of $(Y_n')$, we have
\begin{eqnarray*}
&&\mathbb{P}(A_{h,p}'|Y_0'=l)\\
&&=\mathbb{P}\left(\{\sigma_h'<\infty\}\cap
\{Y_{\sigma_h'}'=h\}|Y_0'=
l\right)\cdot \mathbb{P}(A_{h,p-1}'|Y_0'=h)\\
&&=\cdots\\
&&=\mathbb{P}\left(\{\sigma_h'<\infty\}\cap
\{Y_{\sigma_h'}'=h\}|Y_0'=
l\right)\cdot\left(\mathbb{P}\left(\{\sigma_h'<\infty\}\cap
\{Y_{\sigma_h'}'=h\}|Y_0'=
h\right)\right)^{p-1}\\
&&\quad\cdot \mathbb{P}(A_{h,0}'|Y_0'=h)\\
&&\leq \left(\mathbb{P}\left(\{\sigma_h'<\infty\}\cap
\{Y_{\sigma_h'}'=h\}|Y_0'=
h\right)\right)^{p-1}\cdot \mathbb{P}(A_{h,0}'|Y_0'=h).
\end{eqnarray*}
Then by Lemma \ref{Lemma 3.3(2)'} and (\ref{proof-lemma-3.1-(iii)}), we get
\begin{eqnarray*}
\mathbb{P}(A_{h,p}'|Y_0'=l)\le (Ch^{-\frac{1}{2}+\varepsilon})^p.
\end{eqnarray*}
The proof is complete.\hfill\fbox

\bigskip

\noindent {\bf Proof of Lemma \ref{proof heart2}.} (i) For $p=0$, by (\ref{Ray-Knight representation03-a}) and Lemma \ref{Side-lemma 3}, we have
\begin{eqnarray}\label{u-0-A}
P(C_{h,0}|Y_0=h)=P(\sigma_h=\infty|Y_0=h)<Ch^{-\frac{1}{2}}.
\end{eqnarray}
For $p\ge 1$, by the strong Markov property of $(Y_n)$, we have
\begin{eqnarray*}
P(C_{h,p}|Y_0=h)&=&P\left(\{\sigma_h<\infty\}\cap\{Y_{\sigma_h}=h\}|Y_0=h\right)\cdot P(C_{h,p-1}|Y_0=h)\\
&=&\cdots\\
&=&\left(P\left(\{\sigma_h<\infty\}\cap\{Y_{\sigma_h}=h\}|Y_0=h\right)\right)^p\cdot P(C_{h,0}|Y_0=h).
\end{eqnarray*}
Then by (\ref{u-0-A}) and Lemma \ref{Lemma 3.3(2)}, we get (i).

(ii) For $p=0$, by Lemma \ref{Side-lemma 4}, we have
\begin{eqnarray*}
\sum_{x=1}^{\infty}P(D_{x,h,0}|Z_0=h-1)=\sum_{x=1}^{\infty}P(\tau_h\ge x|Z_0=h-1)=E(\tau_h|Z_0=h-1)<Ch^{\frac12}.
\end{eqnarray*}
Thus (ii) holds in this case. Also by Lemma \ref{Side-lemma 4}, we have
\begin{eqnarray}\label{3.9}
\sum_{x=1}^{\infty}P(D_{x,h,0}|Z_0=h)=\sum_{x=1}^{\infty}P(\tau_h\ge x|Z_0=h)=E(\tau_h|Z_0=h)<Ch^{\frac12}.
\end{eqnarray}
For $p\ge 1$, by the strong Markov property of $(Z_n)$, we have
\begin{eqnarray*}
&&\sum_{x=1}^{\infty}P(D_{x,h,p}|Z_0=h-1)\\
&&=P\left(Z_{\tau_h}=h|Z_0=h-1\right)\cdot
\left(\sum_{x=1}^{\infty}P(D_{x,h,p-1}|Z_0=h)\right)\\
&&=\cdots\\
&&=P\left(Z_{\tau_h}=h|Z_0=h-1\right)\cdot\left(P\left(Z_{\tau_h}=h|Z_0=h\right)\right)^{p-1}\cdot
\left(\sum_{x=1}^{\infty}P(D_{x,h,0}|Z_0=h)\right).
\end{eqnarray*}
Then by Lemma \ref{Branching process}, Lemma \ref{Branching process h-1} and (\ref{3.9}), we have
\begin{eqnarray*}
\sum_{x=1}^{\infty}P(D_{x,h,p}|Z_0=h-1)\le\left(Ch^{-\frac12+\varepsilon}\right)^ph^{\frac12}.
\end{eqnarray*}

(iii) By the definition of $(C_{h,0}')$ (i.e. (2.17)) and Lemma \ref{Side-lemma 3'}, we have
\begin{eqnarray}\label{proof-Lem 3.2}
P(C_{h,0}'|Y_0'=h)=P(\sigma_h'=\infty|Y_0'=h)\leq Ch^{-\frac12}.
\end{eqnarray}
For $p\ge 1$, by the strong Markov property of $(Y_n')$,  we get
\begin{eqnarray*}
&&P(C_{h,p}'|Y_0'=l)\\
&&=P\left(\{\sigma_h'<\infty\}\cap\{Y_{\sigma_h'}'=h\}|Y_0'=l\right)\cdot P(C_{h,p-1}'|Y_0'=h)\\
&&=\cdots\\
&&=P\left(\{\sigma_h'<\infty\}\cap\{Y_{\sigma_h'}'=h\}|Y_0'=l\right)\cdot \left(P\left(\{\sigma_h'<\infty\}\cap\{Y_{\sigma_h'}'=h\}|Y_0'=h\right)\right)^{p-1}\\
&&\quad\cdot P(C_{h,0}'|Y_0'=h)\\
&&\leq \left(P\left(\{\sigma_h'<\infty\}\cap\{Y_{\sigma_h'}'=h\}|Y_0'=h\right)\right)^{p-1}\cdot P(C_{h,0}'|Y_0'=h).
\end{eqnarray*}
Then by Lemma \ref{Lemma 3.3(2)} and (\ref{proof-Lem 3.2}), we get
\begin{eqnarray*}
\mathbb{P}(C_{h,p}'|Y_0'=l)\le C(h^{-\frac12+\varepsilon})^{p}.
\end{eqnarray*}
The proof is complete.\hfill\fbox

\section{Remarks and open questions}

\begin{rem}
Similar to favorite downcrossing site process, we can define favorite upcrossing site process.
By the symmetry of one-dimensional simple symmetric random walk, we can get the corresponding results similar to Theorem \ref{mainthm-0} for favorite upcrossing site process.
\end{rem}

\begin{rem}
Define
$$
\xi^*(x,n):=\max\{\xi_U(x,n),\xi_D(x,n)\}.
$$
We call $x$  a favorite one-side site of the random walk at time $n$  if
$$
\xi^*(x,n)=\sup_{y\in \mathbb{Z}}\xi^*(y,n).
$$
The set of favorite one-side sites of the random walk at time $n$ is denoted by $\mathcal{K}^*(n)$. $(\mathcal{K}^*(n))_{n\ge 1}$ is called the favorite one-side site process of the one-dimensional simple symmetric random walk.

It seems that it is more diffucult to understand the process $(\mathcal{K}^*(n))_{n\ge 1}$ compared to the three processess $(\mathcal{K}_D(n))_{n\ge 1}, (\mathcal{K}(n))_{n\ge 1}$ or $(\mathcal{E}(n))_{n\ge 1}$. As to one-dimensional simple symmetric random walk, we introduce the following questions:
\end{rem}

{\bf Question 1.} With probability 1 there are only finitely many times at which there are at least four favorite one-side sites?

{\bf Question 2.} With probability 1 three favorite one-side sites occurs infinitely often?

{\bf Question 3.} Is the favorite one-side site process transient? If the answer is yes, how about the escape rate?

\bigskip
{ \noindent {\bf\large Acknowledgments}\ \ We thank the referee for helpful comments which helped improve the presentation of the paper. This work was supported by the National Natural Science Foundation of China (Grant Nos. 12171335, 12101429,  12071011, 11931004 and  11871184), the Simons Foundation (\#429343) and the Science Development Project of Sichuan University (2020SCUNL201).

\end{document}